\documentclass[10pt]{amsart}

\usepackage{amsmath}
\usepackage{amssymb}
\usepackage{amsthm}
\usepackage{graphics}
\usepackage{eucal}
\usepackage{bm}
\usepackage{mathtools}
\usepackage{mathrsfs}

\usepackage{color}

\usepackage[pdftex,bookmarks,colorlinks,breaklinks]{hyperref}
\usepackage{mathtools}

\newtheorem{theorem}{Theorem}[section]
\newtheorem{definition}{Definition}
\newtheorem{proposition}{Proposition}[section]
\newtheorem{lemma}{Lemma}[section]
\newtheorem{corollary}{Corollary}[section]

\newtheorem{question}{Question}[section]

\newtheoremstyle{myremark}{10pt}{10pt}{}{}{\scshape}{.}{.5em}{}
\theoremstyle{remark}
\theoremstyle{myremark}
\newtheorem{remark}{Remark}
\usepackage{geometry}

\numberwithin{equation}{section}
\newcommand\spr{Sprind\v{z}uk}

\newcommand{\be}{\mathbf{e}}

\newcommand{\GL}{\operatorname{GL}}

\newcommand{\supp}{\operatorname{supp}}
\newcommand{\cov}{\operatorname{cov}}

\newcommand{\bq}{\mathbf{q}}
\newcommand{\bx}{\mathbf{x}}

\newcommand{\cL}{\mathcal{L}}
\newcommand{\bp}{\mathbf{p}}

\newcommand{\bv}{\mathbf{v}}
\newcommand{\bw}{\mathbf{w}}
\newcommand{\f}{\mathbf{f}}
\newcommand{\rk}{\operatorname{rank}}

\newcommand{\ba}{\mathbf{a}}

\newcommand{\bc}{\mathbf{c}}

\newcommand{\by}{\mathbf{y}}

\newcommand{\bg}{\mathbf{g}}

\newcommand{\Q}{\mathbb{Q}}
\newcommand{\K}{\mathcal{K}}

\newcommand{\FF}{\mathbb{F}_q((T^{-1}))}

\newcommand{\Z}{\mathbb{Z}}
\newcommand{\R}{\mathbb{R}}
\newcommand{\N}{\mathbb{N}}

\newcommand{\inv}{^{\text{-}1}}

\usepackage{mathtools}
\usepackage{tikz-cd}
\usepackage{graphicx}
\begin{document}

\title{ Singular Vectors in Real Affine Subspaces}
\author{Shreyasi Datta and  Yewei Xu}
\address{ Department of Mathematics, University of Michigan, Ann Arbor, MI 48109-1043}
\email{dattash@umich.edu} \email{ywhsu@umich.edu}

\date{}

\begin{abstract}
We prove inheritance of measure zero property of the set of singular vectors for affine subspaces and submanifolds inside those affine subspaces. We define a notion of $n$-singularity for matrices, which is closely related to the uniform exponent of irrationality. For certain affine subspaces, we show that the set of singular vectors has measure zero if and only if the parametrizing matrix is not $n$-singular. In particular, we show for affine hyperplanes the set of singular vectors has measure zero if and only if the parametrizing matrix is not rational. 

\end{abstract}

\maketitle
\section{Introduction}
Diophantine approximation of real vectors by rational vectors starts with the famous Dirichlet's theorem. 
In recent years, the study of \textit{singular vectors} in the classical setting i.e.\ in $\R^n$, and in a real submanifold has drawn a lot of attention (see \cite{C, CC, DFSU, KNW, KW, AGMS, Osama, KKLM} and references therein). Much less is  known about singular vectors compared to its counterparts, \textit{very well approximable} vectors or \textit{badly approximable vectors}; see \cite{BRV} for a survey. We recall the definition of singular vectors, which were originally introduced
by A. Khintchine in the 1920s (see \cite{Kh, Cas}). We fix $\Vert\cdot\Vert$ to be the max norm in $\R^n$ for any $n>1$.

 \begin{definition}
 \label{DEFSING}
A vector $\bx=(x_1,\cdots,x_n)\in \R^n$ is said to be singular if for every $c>0$, for all sufficiently large $Q>0$ there exist $\mathbf{0}\neq\bq\in \Z^n$, $q_0\in\Z$ satisfying the following system of inequalities, \begin{equation}\label{sing}
 \begin{aligned}&\vert \bq\cdot\bx+q_0\vert<\frac{c}{Q^n},\\&\Vert \bq\Vert\leq Q.\end{aligned}
 \end{equation}
\end{definition}
In the late 1960s Davenport and Schmidt showed
(see \cite{DS, DS2}) that the set $$\{t\in \R~|~(t, t^2) \text{ is singular}\}$$ has Lebesgue measure zero. This problem is significantly difficult than studying singular vectors in Euclidean space due to the dependency of coordinates. This kind of problems boost a momentum after Kleinbock and Margulis proved the famous \textit{{\spr} conjecture} in 1998 \cite{KM}. This conjecture states that for an analytic submanifold which is not contained in any affine subspace of $\R^n$, the set of \textit{very well approximable} vectors has measure zero. Kleinbock and Margulis translated the problem into a dynamical problem and proved  \textit{quantitative nondivergence} to address the dynamical problem. The correspondence between Diophantine approximation and homogeneous dynamics was first realized by Dani, in \cite{Dani}. Then onwards, using the similar philosophy there has been a major improvement in the study of Diophantine approximation in manifolds; see \cite{BKM, BBKM, KW, KT}. It was shown in  \cite{KW}, that the set of singular vectors in a  submanifold of $\R^n$, which is not contained inside any affine subspace has Haar measure zero. \\

We now state a special case of our main theorem (Theorem \ref{sing_main_1}).
	\begin{theorem}\label{push_main}
  Let $\cL$ be an affine subspace in $\R^n$. Let $\f:U\subset \R^d\to \cL\subset \R^n$ be an analytic map such that $\f(B)$ is not contained in any proper affine subspace of $\mathcal{L}$ for any $B$ ball in $U$. Suppose that $\lambda_{U}$ and $\lambda_{\mathcal{L}}$ are Haar measures on $U$ and $\mathcal{L}$ respectively. Then the following are equivalent:
   \begin{enumerate}
    \item\label{1*} For $\lambda_U$-almost every $\bx$, $\f(\bx)$ is not singular. 
    \item\label{2*} for $\lambda_{\cL}$-almost every $\by$, $\by$ is not singular.
    \end{enumerate}
   \end{theorem}
  The main theorem (Theorem $1.1$) in \cite{KW} implies that if $\cL=\R^n$, then $(\ref{1*})\Leftrightarrow (\ref{2*})$. So our theorem generalizes Theorem $1.1$ in \cite{KW}  for any affine subspace $\cL$.
    
   We have the following corollary.
   \begin{corollary}
  Let $\cL$ be an affine subspace and $\f:U\subset \R^d\to \cL\subset \R^n$ be an analytic map such $\f(B)$ is connected. Suppose that $\f(B)$ is not contained in any proper affine subspace of $\mathcal{L}$ for any $B$ ball in $U$. Suppose that $\lambda_{U}$ and $\lambda_{\mathcal{L}}$ are Haar measures on $U$ and $\mathcal{L}$ respectively. Then the following are equivalent:
   \begin{enumerate}
   \item\label{1} There exists $\bx\in U$ such that $\f(\bx)$ is not singular.
   \item\label{2} There exists $\by\in \cL$ that is not singular.
    \item\label{3} For $\lambda_U$-almost every $\bx$, $\f(\bx)$ is not singular. 
    \item\label{4} for $\lambda_{\cL}$-almost every $\by$, $\by$ is not singular.
    \end{enumerate}
   \end{corollary}
   
   \begin{proof} Theorem \ref{push_main} in this paper implies $(\ref{3})\Leftrightarrow (\ref{4})$. Theorem $1.4$ in \cite{Kll} implies that $(\ref{1})\Leftrightarrow (\ref{3})$ and $(\ref{2})\Leftrightarrow (\ref{4})$.  Thus our theorem completes the picture by showing that all of them are equivalent for affine subspaces and some submanifolds inside them. 
   \end{proof}
    We are motivated by the following question:
\begin{question}\label{q0}
	Find all submanifolds in $\R^n$ such that Haar almost every vector in the submanifold is not singular. 
\end{question}
   Our Theorem \ref{push_main} shows that it is enough to answer the above question for affine subspaces. Following a recent paper \cite{Schlei}, we define a notion of $\omega$-singularity as follows:
\begin{definition}\label{nsingular}
Let $A$ be a $k\times l$ matrix and $\omega>0$. We call $A$ to be  $\omega$-singular if for every $c>0$, for all sufficiently large $Q>0$ there exist $\mathbf{0}\neq\bq\in \Z^l$, $\bp\in\Z^k$ satisfying the following system of inequalities, \begin{equation}\label{singAn}
 \begin{aligned}&\Vert A\bq +\bp\Vert<\frac{c}{Q^\omega},\\&\Vert \bq\Vert\leq Q.\end{aligned}
 \end{equation}
\end{definition}
An $k\times l$ matrix $A$ is called singular if $A$ is $\frac{l}{k}$-singular. The notion of $\omega$-singularity is only slighty different than the so called $\hat\omega$ (see Definition \ref{omegahat}) of a matrix, see Lemma \ref{omegahatsingular}.
   We now state our next main theorem that answers the above question for certain affine subspaces, in particular for affine hyperplanes.
   \begin{theorem}\label{main3}
   Let $\cL$ be a $s$-dimensional affine subspace in $\R^n$ and $A$ be the $(s+1)\times (n-s)$ parametrizing matrix of $\cL$ (see \S\ref{para}).
   Suppose that either (a) all the columns or (b) all the rows of A are rational multiples of one column (resp one row). Then $A$ is not $n$-singular if and only if  $\lambda_{\cL}$-a.e every point in $\cL$ is not singular.
   \end{theorem}
   By definition \ref{nsingular}, the parametrizing matrix $A$ being $n$-singular implies it is singular, since $\frac{n-s}{s+1}<n$. Hence $A$ being not $n$-singular is much weaker assumption than $A$ being not singular. Moreover, in Lemma \ref{trivial}, we show that for $s=n-1$, if the parametrizing matrix $A$ is $n$-singular then $A\in\Q^n$. Hence
 as a corollary of Theorem \ref{main3}, we have the following:
 \begin{corollary}\label{hyperplane}
  Let $\cL$ be an affine hyperplane in $\R^n$ and $A$ be the $n\times 1$ parametrizing matrix of $\cL$ (see \S\ref{para}). Then $A\notin\Q^n$ if and only if $\lambda_{\cL}$-a.e every point in $\cL$ is not singular. 
 \end{corollary}
 
 In a recent paper \cite{KSSY}, it was shown that Lebesgue-almost every point on the planar line $y = ax+b$ is not \textit{Dirichlet-improvable} if and only if $(a,b)\notin\Q^2$. By the definition of Dirichlet-improvability (see \cite{DavenportSch69})  singular vectors are Dirichlet-improvable. Hence for $n=2$ their result implies our Corollary \ref{hyperplane}.
    
 We note that for affine subspaces $\cL$ of dimension less than $n-1$, the condition on $A$ as in Theorem \ref{main3} is not as simple as in the case of affine hyperplanes. Using \cite[Theorem 3.8]{DFSU} and Lemma \ref{omegahatsingular}, the packing dimension of $n$-singular parametrizing matrices $A$ for $s\leq n-2$, can be shown to be positive, but stricly less than $(s+1)(n-s)$.

 We conclude this section by stating the first main theorem of this paper in its full generality and then briefly discussing about methods of proof.
 For definitions of \textit{Federer measure, Besicovitch space, good map, nonplanarity} and \textit{nondegeneracy in affine subspaces}, the readers are referred to \S \ref{sev}.
\begin{theorem}\label{sing_main_1}
	Let $\mu$ be a Federer measure on a Besicovitch metric space $X, \mathcal{L}$ be an affine subspace of $\R^n$, and let $\f: X\to \mathcal{L}$ be a continuous map such that $(\f,\mu)$ is good and nonplanar in $\mathcal{L}.$ Let $\f_\star\mu$ be the pushforward measure and $\lambda_{\mathcal{L}}$ be the Haar measure on $\mathcal{L}$. Then the following are equivalent:
	\begin{enumerate}
	\item\label{e11} There exists one $x\in \supp(\mu)$ such that $\f(x)$ is not singular.
	\item\label{e2} There exists one $\by\in\mathcal{L}$ which is not singular. 
	\item\label{e3} For $\lambda_{\mathcal{L}}$ almost  every $\by\in\mathcal{L}$, $\by$ is not singular.
	\item\label{e4} For $\mu$-almost every $x\in X$, $\f(x)$ is not singular.
	\end{enumerate}
	\end{theorem}
		The main tool to establishing Theorems \ref{sing_main_1} is  a sharpened \textit{quantitative nondivergence} in \cite{Kleinbock-exponent} (Theorem $2.2$). It is important for us that the second hypothesis in Theorem $2.2$ in \cite{Kleinbock-exponent} is weaker than its previous version in \cite{KM}.  We consider the space of unimodular lattices in $\R^{n+1}$ i.e. $\mathrm{SL_{n+1}(\R)/SL_{n+1}(\Z)}$. We take the point $u_{\bx}\Z^{n+1},$  where $u_{\bx}$ is an \textit{unipotent} element as defined in Section \ref{Dyna} and let $g_k$ be a diagonal flow as defined in Section \ref{Dyna}. The \textit{Dani correspondence} links the singularity property of $\bx$ with the divergence of $u_{\bx}\Z^{n+1}$ under the diagonal flow $g_k$, as $k\to \infty$. Both side of this correspondence is crucial to prove Theorem \ref{sing_main_1}. Following ideas from \cite{Kleinbock-exponent, KM, Kleinbock-extremal}, we used quantitative nondivergence to get measure estimates  of the divergent orbits (Theorem \ref{Thm_sing}). We then show that second hypothesis in Theorem \ref{Thm_sing} is necessary. This idea goes back to the work of Kleinbock in \cite{Kleinbock-exponent}, where he dealt with \textit{Diophantine exponents}. 
		
		In order to prove Theorem \ref{main3}, we study approximation of a matrix by higher rank subgroups  and how it effects Diophantine property of the parametrizing matrix. We use estimates from \cite{Kleinbock-exponent} and analyse them in order to apply in our set-up. The novelty of Theorem \ref{main3} is that the sufficient and necessary condition on the parametrizing $A$ is that it is not $n$-singular, which is weaker assumption than $A$ being not singular.

		Our method is very much influenced by the work of Kleinbock 
		in \cite{Kleinbock-exponent}. We show in this paper that the techniques developed in his paper can be adopted in the setting of singular vectors.

\section{Corollaries of Theorems \ref{sing_main_1}}
In Theorem \ref{sing_main_1}, if we take $X=U\subset \R^d$  and $\f$  is a smooth \textit{nondegenerate} map in $\cL$, $M=\f(U)$ then the following corollary follows. Let $\lambda_{\cL}$ be the Haar measure on $\cL$.

\begin{corollary} 
Let $\cL$ be an affine subspace of $\R^n$, and let $M$ be a submanifold of $\cL$ which is nondegenerate in $\cL$. Suppose that $\lambda_{\cL}$ and $\lambda_{M}$ are Haar measures on $\cL$ and $M$ respectively. Then the following are equivalent:
\begin{itemize}
    \item There exists $\by\in M$ such that $\by$ is not singular.
    \item There exists $\by\in \cL$ such that $\by$ is not singular.
    \item For $\lambda_{\cL}$-almost every $\by\in\cL$, $\by$ is not singular.
    \item For $\lambda_{M}$-almost every $\by\in M$, $\by$ is not singular.
\end{itemize}
\end{corollary}
In \cite{KLW}, a class of measures called as `friendly measures' was introduced, which  contains Lebesgue measure, fractal measures,  smooth measures on nondegenerate manifolds and many measures naturally arising from geometric constructions. In particular, if $\mu$ is friendly then $(\mathrm{Id},\mu)$ is good and nonplanar in $\R^n$.  Many natural measures coming from geometric constructions can be shown to
have an even stronger property. These measures were referred to as ‘absolutely
decaying and Federer’ in \cite{KLW} and as ‘absolutely friendly’ in \cite{PV}. If $\mu$
is absolutely decaying, Federer and $\f$ is nondegenerate at $\mu$-almost every point of $\R^d$,
then $(\f, \mu)$ is good and nonplanar; see \cite{KLW} \S $7$. Hence the following corollaries follow from Theorem \ref{sing_main_1}. 
  
  \begin{corollary}\label{cor200}
Let $\cL$ be a $d$-dimensional affine subspace of $\R^n$, and $\mu$ be a  (absolutely) friendly measure on $\R^d$. Let $\f:\R^d\to \cL$ be an affine isomorphism. Then the following are equivalent:
\begin{itemize}
\item There exists one $\by\in\supp(\mu)$ such that $\f(\by)$ is not singular.
    \item There exists one $\by\in\cL$ which is not singular.
    \item For ${\lambda}_{\cL}$-almost every $\by\in\cL$, $\by$ is not singular.
 \item For $\mu$-almost every $\by$, $\f(\by)$ is not singular .\end{itemize}
\end{corollary}
\begin{corollary}\label{cor100}
Let $\mu$ be an absolutely decaying and Federer measure on $\R^d$, $\cL$ is an affine
subspace of $\R^n$, and let $\f:\R^d\to \cL$ be a smooth map which is nondegenerate in $\cL$
at $\mu$-almost every point of $\R^d$. Then the following are equivalent:
\begin{itemize}
\item There exists one $\by\in\supp(\mu)$ such that $\f(\by)$ is not singular.
    \item There exists one $\by\in\cL$ which is not singular.
    \item For ${\lambda}_{\cL}$-almost every $\by\in\cL$, $\by$ is not singular.
 \item For $\mu$-almost every $\by$, $\f(\by)$ is not singular .\end{itemize}
\end{corollary}
\section{Several preliminaries}\label{sev}
\subsection{Besicovitch space}\label{sec3.1}
   A metric space $X$ is called \emph{Besicovitch} \cite{KT} if there exists a constant $N_X$ such that the following holds: for any bounded subset $A$ of $X$ and for any family $\mathcal{B}$ of nonempty open balls in $X$ such that
   every  $x \in A$  is a center of some ball in $\mathcal B,$
 there is a finite or countable subfamily $\{B_i\}$ of $\mathcal B$ with
 $$ 1_A \leq \sum_{i}1_{B_i} \leq N_X. $$ By \cite{Al}, we know $\R,\Q_\nu$ and $\FF$ are Besicovitch spaces. 
 
\subsection{Federer measure} Let $X$ be a metric space. We define $D$-\textit{Federer} measures following \cite{KLW}. Let $\mu$ be a Radon measure on $X$, and $U$ an open subset of $X$ with $\mu(U) > 0$. We say that $\mu$ is \textit{$D$-Federer on $U$} if
$$ \sup_{\substack{x \in \supp \mu, r > 0\\ B(x, 3r) \subset U}} \frac{\mu(B(x, 3r))}{\mu (B(x,r))} < D.$$
We say that $\mu$ as above is \textit{Federer} if, for $\mu$-almost every $x \in X$, there exists a neighbourhood $U$ of $x$ and $D > 0$ such that $\mu$ is $D$-Federer on $U$. We refer the reader to \cite{KLW} and \cite{KT} for examples of Federer measures. 

 \subsection{ Nondegeneracy in an affine plane}
\label{nondeg}
 \indent Let $U$ be an open subset of $\R^d$, and $\cL$ be an affine subspace of $\K^n$. 
 Recalling \cite{Kleinbock-exponent}, a differentiable map $\f : U \to \cL\subset \R^n$ is said to be  \textit{nondegenerate in $\cL$} of $\R^n$ at $\bx \in U$ if the span of all the partial derivatives of $\f$ at $\bx$ up to some order coincides with
the linear part of $\cL$. If $M$ is a $d$-dimensional submanifold of $\cL$, we will say that
$M$ is nondegenerate in $\cL$ at $y\in M$ if any (equivalently, some) diffeomorphism $\f$
between an open subset $U$ of $\R^d$ and a neighborhood of $y$ in $M$ is nondegenerate
in $\cL$ at $\f^{-1}(y)$. We will say that $\f : U \to \cL$ (resp., $M \subset \cL$) is nondegenerate in $\cL$ if it is nondegenerate in $\cL$ at $\lambda_U$-almost every point of $U$, where $\lambda_U$ is the Haar measure on $U$ (resp., of $M$, in the sense of the
smooth measure class on $M$).  
  
 \noindent 
  \subsection{Nonplanar in an affine plane}
Let $X$ be a metric space, $\mu$ is a measure on $X$ and let $\f:X\to \mathcal{L}$, where $\mathcal{L}$ is an affine subspace in $\R^n$. We will say $(\f,\mu)$ is nonplanar in $\mathcal{L}$ if  for any ball $B$ with $\mu(B)>0$, $\mathcal{L}$ is the intersection of all affine subspaces that contain $\f(B\cap\supp{\mu})$. If an analytic map $\f : U \to \cL$ is nondegenerate in $\cL$ then $(\f,\lambda_U)$ is nonplanar in $\cL$.\\
\subsection{Good maps}
Let $X$ be a normed space and $\mu$ be a measure in $X$.
For $A \subset X$ with $\mu(A)>0$ and $f$ a $\R$-valued function on $X$, denote $$\Vert f\Vert_{\mu,A}:=\sup_
{x\in A\cap\,\supp\mu}\vert f(x)\vert.$$
A function $f:X\to \R$ is called $(C,\alpha)$-good on $U\subset X$ with respect to $\mu$ if for any open ball $B\subset U$ centered in $\supp(\mu)$ and $\varepsilon>0$ one has $$
\mu\left(\{x\in B~|~ \vert f(x)\vert <\varepsilon\}\right)\leq C\left(\frac{\varepsilon}{\Vert f\Vert_{\mu, B}}\right)^{\alpha}\mu(B).$$
Next, if $\f=(f_1,\cdots,f_n): X\to \R^n$ be a map, and $\mu$ be a measure on $X$, we call $(\f,\mu)$ is $(C,\alpha)$-good at $x\in X$ if there exists a neighborhood $U$ of $x$ such that any linear combination of $1,f_1,\cdots,f_n$ is $(C,\alpha)$-good on $U$ with respect to $\mu$. We will  say $(\f,\mu)$ is good if it is good at $\mu$-a.e point. When $\mu$ is the Haar measure, we omit `with respect to $\mu$' and just say that the function $\f$ is good at $x$. Polynomials are $(C,\alpha)$ good functions at every point, see \cite[Lemma $4.1$]{DM}.

 \noindent In fact, there are many examples of good functions. We recall Lemma $2.5$ of \cite{KM} which shows that nondegenerate functions are good.
 \begin{proposition}  Let $\f = (f_1,\cdots, f_n)$ be a $C^l$ map from an open subset $U \subset \R^d$
to $\R^n$
which is $l$-nondegenerate in $\R^n$ at $\bx_0 \in U$. Then there is a neighborhood $V \subset U$ of $\bx_0$
such that any linear combination of $1, f_1\cdots f_n$ is 
$(C',\alpha)$-good on
$V$, where $C',\alpha>0$ only depend on $d,l$ and $\R$. In particular, the nondegeneracy of $\f$ in $\R^n$ at $\bx_0$  implies that $\f$ is good at $\bx_0$.
\end{proposition}
\noindent Moreover, Corollary $3.2$ in \cite{Kleinbock-extremal} says the following:
\begin{corollary}
 Let $\cL$ be an affine subspace of $\R^n$ and let $\f = (f_1,\cdots, f_n)$
be a smooth map from an open subset $U$ of $\R^d$ to $\cL$ which is nondegenerate
in $\cL$ at $\bx_0 \in U$. Then $\f$ is good at $\bx_0$.
\end{corollary}
\subsection{Quantitative Nondivergence}         
       If $\Delta$ is a $\Z$-submodule of $\R^{n+1}$, we denote by $\R\Delta$ its $\R$ linear span inside $\R^{n+1}$, and define the rank of $\Delta$ by
$$\rk(\Delta):= \dim_{\R}(\R\Delta).$$ Any discrete $\Z$-submodule of $\R^{n}$ will be of the form $\Z \bx_1+\cdots+\Z \bx_r$, where $\bx_1,\cdots,\bx_r$ are linearly independent over $\R$. Following \cite{KM}, we say that $\Delta\subset \Z^{n}$, a $\Z$-submodule, is primitive if $$\Delta=\R\Delta \cap \Z^{n}.$$ 
 \noindent
  Following notation from \cite{KT} $ \S 6.3$, let 
 $$\mathfrak P({\Z},n)\coloneqq \text{ the set of all nonzero primitive submodules of } \Z^{n}. $$ Let $\be_0,\be_1,\cdots,\be_n\in \R^{n}$ be the standard basis of $\R^{n}$ over $\R$.
  Then we can define the standard basis of $\bigwedge^j \R^{n}$  over $\R$ to be be  $\{\be_I=\be_{i_1}\wedge\cdots\wedge \be_{i_j}~|I\subset\{0,\cdots, n\}\text{ and } i_1<i_2<\cdots<i_j \}$. We can extend the norm in $\R^{n}$ to a norm in the exterior algebra $\bigwedge^j \R^{n}$. Namely,
  for an element $\ba=\sum a_I \be_I \in \bigwedge^j \R^{n} $, we define $\Vert \ba\Vert:= \max_{I} \vert a_I\vert$. For a $\Z$-submodule $\Delta=\Z \bw_1+\cdots +\Z \bw_r$ of $\Z^{n}$, we can define $$
  \cov(\Delta):=\Vert \bw_1\wedge \cdots\wedge \bw_r \Vert.$$ We use the same notation $\Vert \cdot \Vert$ for supremum norm in $\R^n$, but it should be clear from the context.\\
 
   \indent Let us recall Theorem $2.2$ from \cite{Kleinbock-exponent}, which is an improvement to one of the main theorems in \cite{KM}. The following theorem is referred as quantitative nondivergence as this quantifies the nondivergence results of Margulis in \cite{mar1}. 
   \begin{theorem}\label{QND2_R}
		 Let  $m\in \N$, $C,\alpha>0$, $0<\rho<1$ and $X$ is a Besicovitch space with $N_X$ being the Besicovitch constant. Let $B = B(x_0,r_0)\subset X$ and a continuous map $h: \tilde{B}:=B(x_0,3^mr_0)
		\to \GL_m(\R)$ be given and $\mu$ be a $D$-Federer measure on $\tilde{B}$.
		Suppose that the following conditions are satisfied by the function $h$ and the measure $\mu$.
		
		{\rm(i)} For every $\,\Delta\in \mathfrak{P}({\Z,m})$, the function
		$\cov( h(\cdot)\Delta)$ is $(C,\alpha)$-good \ on $\tilde{B}$  with
		respect to $\mu$;
		
		{\rm(ii)}  For every $\,\Delta\in \mathfrak P({\Z,m})$,
		$\sup\limits_{x\in B\cap\, \supp\mu} \cov(h(x)\Delta) \ge \rho^{\rk(\Delta)}$.
		
		Then 
		for any  positive $ \varepsilon\le
		\rho$, one has
		$$
		\mu\left(\big\{x\in B\bigm| \delta \big(h(x)\Z^m\big) < \varepsilon
		\big\}\right)\le mC (N_XD^2)^m
		\left(\frac{\varepsilon} {\rho} \right)^\alpha \mu(B).\
		$$
	
	\end{theorem}
		\subsection{Dynamics}
	\subsubsection{Homogeneous space}\label{Dyna}
We consider $\Omega_{n+1}:=\mathrm{SL_{n+1}(\R)/SL_{n+1}}(\Z)$ which is noncompact and it has finite volume (c.f. \cite{JS}). Moreover, $\Omega_{n+1}$ can be indentified with the space of covolume $1$ lattices in $\R^{n+1}$. 
\subsubsection{The smallest vector} Let us define the \textit{ length of the smallest vector} in a lattice $\Delta$ in $\K^{n+1}$, $$\delta(\Delta):=\inf\{\Vert \bx\Vert~|~ \bx\in \Delta\setminus\{0\}\}.$$\\

We recall the following theorem by Mahler, that describes compact sets of the space $\Omega_n$. 
\begin{theorem}\label{Mah}
The set $Q_{\varepsilon}=\{\Delta\in\Omega_n~|~ \delta(\Delta)\geq \varepsilon \}\}$ is compact for all $\varepsilon>0$.
\end{theorem}

\subsubsection{Flows in the homogeneous spaces}
 For $\by\in \R^n$, let us denote $u_{\by}= \begin{bmatrix} 1 & \by\\
0 & 1
\end{bmatrix}$ and $\Lambda_{\by}=u_{\by}\Z^{n+1}$. We also consider the diagonal flows as follows. We take a diagonal flow $$g_k=\begin{bmatrix}
		e^{nk} & 0 &\cdots & 0\\
		0 & e^{-k}  &\cdots &0\\
	    \vdots & \vdots & \vdots   & 0\\
	    0 & 0 & \cdots & e^{-k}
		\end{bmatrix}, \text{ where } k\in \N.$$\\
	
	 \subsection{Connecting number theory and dynamics}
 The following connection, often referred as Dani correspondence, between number theory and dynamics has been explored in several breakthrough works; see \cite{KM, DFSU}. 
 \begin{lemma}[Dani Correspondence]
 \label{Dani}
   A vector $\bx\in \R^n$ is singular if and only if the corresponding trajectory $\{g_ku_\bx \Z^{n+1}\mid k\geq 0\}$ is divergent in $\Omega_{n+1}$. 
  \end{lemma}
Using Theorem \ref{QND2_R} and Lemma \ref{Dani}, we have the following theorem:
   \begin{theorem}\label{Thm_sing}
   Let $X$ be a Besicovitch space, $B=B(x,r)\subset X$ a ball, $\mu$ be a Federer measure on $X$, and suppose that $\f:B\to \R^n$ is a continuous map. We assume that the following two properties are satisfied.
	\begin{enumerate}
		\item\label{c11} There exists $C,\alpha>0$ such that all the functions $x\to \cov( g_ku_{\f(x)}\Gamma)$, $\Gamma\in\mathfrak{P}({\Z,n+1})$, are $(C,\alpha)$ good on $\tilde B$ w.r.t. $\mu$;
		\item\label{c22} there exists $c>0$ and $k_i\to \infty$ such that for any $\Gamma\in\mathfrak{P}({\Z,n+1})$ one has 
		\begin{equation}\label{2nd_cond_R}
		\sup_{B\cap\supp{\mu}}\cov( g_{k_i}u_{\f(x)}\Gamma)\geq c^{\rk({\Gamma})}.
		\end{equation}
	\end{enumerate}
Then $$\mu\{x\in B~|~\f(x) \text{ is singular }\}=0.$$
\end{theorem}
  The difference between the above theorem and the main theorem in \cite{KW} is that the second condition in the above theorem is weaker, and therefore the theorem is stronger. \\
  
The following lemma follows from Minkowski's convex body theorem.
\begin{lemma}\label{p33}
 For any $m>0$, $g\in \mathrm{SL(m,\Z)}$, $\Delta \in \mathfrak{P}({\R,m})$, we have that 
$$ \delta (g\Z^m)\leq 2\cov( g\Delta)^{\frac{1}{\rk(\Delta)}}. $$ 
\end{lemma}
\begin{proposition}\label{p2}
	The second Condition \ref{2nd_cond_R} in Theorem \ref{Thm_sing} is necessary. In fact, if the Condition does not hold, then
	$\f(\supp{\mu}\cap B)$ is contained in the set of singular vectors.
\end{proposition}
\begin{proof}
	If the second condition does not hold then for every $c>0$ and for all large enough $k$, there exists $\Gamma_k\in \mathfrak{P}({\Z,n+1})$ such that $$
		\sup_{B\cap\supp{\mu}}\cov (g_k u_{\f(x)}\Gamma_k)<c^{\rk(\Gamma_k)}.
		$$ Hence by Lemma \ref{p33} for $\mu$-almost every every $x\in B$, $\delta(g_k u_{\f(x)}\Z^{n+1})\leq \cov( g_k u_{\f(x)}\Gamma_k )^{\frac{1}{\rk(\Gamma_k)}}\implies \delta(g_ku_{\f(x)}\Z^{n+1})< c$ for all sufficiently large $k$. Now using Lemma \ref{Dani}, we can conclude that for $\mu$-almost every $x\in B$ we have $\f(x)$ to be singular.
\end{proof}
\begin{remark}
Note that in the previous proposition, the other side of Dani correspondence i.e. Lemma \ref{Dani}, play a crucial role. 
\end{remark}

\subsection{Covolume calculation}
  
Let us denote the set of rank $j$  submodules of ${\Z}^{n+1}$ as $\mathcal{S}_{n+1, j}$. 
\begin{lemma}\label{lm1}
Let $\bw\in\bigwedge^j(\Z^{n+1}),\bw=\sum w_I\be_I$. Then 
\begin{equation}
g_k u_\bx\bw= =e^{-jk}\sum_{\lbrace I,~ 0\notin I\rbrace} w_I\be_I+e^{(n-j+1)k}\sum_{\lbrace I~|~ 0\in I\rbrace}\left(w_I+\left(\sum_{i\notin I}\pm w_{I\setminus\{0\}\cup\{i\}}x_i\right)\right)\be_I.
\end{equation}
\end{lemma}
\begin{proof}
Note that  $u_{\bx}$ leaves $\be_0$ invariant and sends $\be_i$ to $x_i\be_0+\be_i$ for $i\geq 1$.
Therefore 
$$
u_{\bx}(\be_I)=\left.
\begin{cases}
\be_I  &\text{ if } 0\in I\\
\be_I +\sum\limits_{i\in I} \pm x_i\be_{I\setminus\{i\}\cup\{0\}} & \text{ if } 0\notin I.
\end{cases}
\right.
$$
 \noindent  
 Also, under the diagonal action of $g_k$, the vectors $\be_i$ are eigenvectors with eigenvalue $e^{-k}$ for $i\geq 1$ and $\be_0$ is an eigenvector with eigenvalue $e^{nk}$. Therefore,
\begin{equation}
g_k u_{\bx}(\be_I)=\begin{cases}
\begin{aligned}
& e^{(n-j+1)k}\be_I  &\text{ if } 0\in I\\
& e^{-jk}\be_I\pm e^{(n-j+1)k}\sum\limits_{i\in I}  x_i\be_{I\setminus\{i\}\cup\{0\}} & \text{ if } 0\notin I.
\end{aligned}
\end{cases}
\end{equation}

Thus, for $\bw\in\bigwedge^j(\Z^{n+1}),\bw=\sum w_I\be_I$ with $w_I\in\Z$ the conclusion follows.
\end{proof}
  Let $V_0$ be the $\Z$ submodule of $\R^{n+1}$ generated by $\be_1,\cdots,\be_n$. Note we can write \begin{equation}\label{cw}
 \bc(\bw)=\begin{bmatrix}
\bc(\bw)_0\\
\bc(\bw)_1\\
\vdots\\
\bc(\bw)_n
\end{bmatrix}, \end{equation}
where $\bc(\bw)_i=\sum\limits_{\substack{J\subset\{1,\cdots,n\} \\\# J=j-1}} w_{J\cup\{i\}}\be_J\in \bigwedge^{j-1}(V_0)$  for $i=0,\cdots, n$. Let $\tilde\bx=(1,\bx)$ and $\pi$ be the orthogonal projection from $\bigwedge ^j(\R^{n+1})\to \bigwedge^j V_0$.

\begin{proposition}\label{p3}
The second Condition \eqref{c22} in Theorem \ref{Thm_sing} is equivalent to the following condition. 
\begin{enumerate}
\item[(2')]\label{2prime} There exists $c>0$ ~
and  $~k_i\to\infty$  such that  $\forall~j=1,\cdots, n$ and  $\forall~\bw\in \bigwedge^{j}\Z^{n+1}$, one has  \begin{equation}\label{c222}
\max\bigg( e^{(n-j+1)k}\Vert R\bc(\bw) \Vert, e^{-jk}\Vert \pi(\bw)\Vert\bigg)\geq c^{j},
\end{equation}
\end{enumerate}
where $R$ is a $(s+1)\times (n+1)$ matrix that depends on the ball $B$, the function $\f$ and the measure $\mu$.
\end{proposition}
\begin{proof}
By Lemma \ref{lm1}, we can write
\begin{equation}\begin{aligned}
& \ \ \ g_k u_{\bx}\bw \\
& = e^{(n-j+1)k}\left(\be_0\wedge\sum_{i=0}^n x_i\bc(\bw)_i\right)+ e^{-jk}\sum_{\lbrace I\,|\,0\notin I\rbrace}w_I\be_I\\
& =e^{(n-j+1)k}\left(\be_0\wedge\tilde{\bx}\cdot\bc(\bw)\right)+ e^{-jk}\pi(\bw).\end{aligned}
\end{equation}
Hence, we have 
$$\begin{aligned}
\cov( g_k u_{\bx}\Delta) &=  \Vert g_k{u_\bx}\bw\Vert \\&= \max\bigg( e^{(n-j+1)k}\Vert\sum_{i=0}^n x_i \bc(\bw)_i\Vert, e^{-jk} \Vert \pi(\bw)\Vert\bigg).\\
\end{aligned}
$$
Thus for $\bx=\f(x)$,
\begin{equation}
\label{covolume}
	\begin{aligned}
	&\sup_{x\in B\cap\, \supp\mu} \cov( g_k u_{\f(x)}\Delta) \\
	&=\max\left( e^{(n-j+1)k}\sup_{x\in B\cap\, \supp\mu}\Vert \tilde\f(x)\cdot\bc(\bw) \Vert, e^{-jk}\Vert \pi(\bw)\Vert\right),
	\end{aligned}
\end{equation} where $\tilde{\f}=(1,f_1,\cdots,f_n)$ and $B$ is a ball in $\R^d$. Now suppose that the $\R$-span of the restrictions of $1, f_1, \cdots, f_n$ to $B\cap \supp\mu$ has dimension $s+1$ and choose $g_1, \cdots, g_s :B\cap\supp \mu\mapsto \R$ such that $1, g_1, \cdots, g_s$ form a basis of the space.  Therefore there exists a matrix $R=(r_{i,j})_{(s+1)\times(n+1)} $ such that $\tilde{\f}(x)=\tilde{\bg}(x)R ~ \forall~ x\in B\cap \supp \mu$ where $\tilde{\bg}=(1, g_1,\cdots, g_s)$. We can rewrite \begin{equation}\label{e1}
    \sup_{x\in B\cap \supp\mu}\Vert \tilde\f(x)\bc(\bw) \Vert=\sup_{x\in B\cap \supp\mu}\Vert \tilde\bg(x)R\bc(\bw) \Vert.\end{equation} Hence by \eqref{covolume} and \eqref{e1} the second Condition \eqref{c22} in \ref{Thm_sing} and is equivalent to the Condition $(2)'$ \eqref{c222} by equivalence of norms since $1, g_1, \cdots, g_s$ are linearly independent.
\end{proof}

\begin{lemma}\label{l1} For all $(\f,\mu)$  nonplanar in $\mathcal{L}$, the matrix $R$ in Condition \eqref{c22} can be chosen uniformly for all ball $B$ with $\mu(B)>0$.
\end{lemma}
\begin{proof}
Let $\dim\mathcal{L}=s$ and let
\begin{equation}{\label{R-type}}\mathbf{h} : \R^s\to \mathcal{L}\subset \R^n \text{ be an affine isomorphism, and }\tilde{\mathbf{h}}(\bx)= \tilde{\bx}R, \bx\in \R^s,
\end{equation} 
where $\tilde{\mathbf{h}}:=(1,h_1, \cdots, h_n) $. Then $\bg= \mathbf{h}\inv\circ \f$ spans the space of $F$-span of the restrictions of $1, f_1, \cdots, f_n$ to $B\cap \supp\mu$ and satisfies $\tilde{\f}(x)=\tilde{\mathbf{g}}(x)R ~\forall ~x\in B\cap\supp\mu$. Moreover, since $(\f,\mu)$ is nonplanar in $\mathcal{L}$, we have $1,g_1,\cdots,g_s$ are linearly independent. Since we chose $R$ such as it only depends on the affine subspace $\mathcal{L}$, Condition $(2)'$ \eqref{c222} only depends on the subspace $\mathcal{L}$ as long as $(\f,\mu)$ is nonplanar in $\mathcal{L}$.
\end{proof}
\subsection{Completing the proof of Theorems \ref{sing_main_1}}
\subsubsection{\eqref{e3}  $\iff$ \eqref{e4} in Theorem \ref{sing_main_1}}
By Proposition \ref{p2} if \eqref{e3} is true, then Condition \eqref{c22}${\iff}$ Condition $(2)'$ (see $\eqref{c222}$) due to Proposition \ref{p3}, should be satisfied. By the hypothesis of Theorem \ref{sing_main_1}, $(\f,\mu)$ is good and therefore by Theorem \ref{Thm_sing}, \eqref{e4} is true. Since $(\mathbf{Id},\lambda_{\mathcal{L}})$ is nonplanar in $\mathcal{L}$, same argument as above shows \eqref{e4} implies \eqref{e3} due to Lemma \ref{l1} .
\subsubsection{\eqref{e11} $\implies$ \eqref{e2} in Theorem \ref{sing_main_1}}
Since $\f(x)\in \cL$ for all $x\in B\cap\supp(\mu)$, \eqref{e11} $\implies$ \eqref{e2}.
\subsubsection{\eqref{e2} $\implies$ \eqref{e4} in Theorem \ref{sing_main_1}}
Let \eqref{e2} be true, that is there is one $\by\in\cL$ such that $\by$ is not singular.  Then by Proposition \ref{p2}, the second condition in Theorem \ref{Thm_sing} is true. Since $(\f, \mu)$ is good, we have that $\mu$-almost every $x\in X$ has a neighbourhood $V$ such that $(\f, \mu)$ is good. Since $\mu$ is Federer, and we can choose $V$ such that $\mu$ is $D$-Federer on $V$ for some $D>0$. Let us choose a ball $B= B(x, r) $ of positive measure such that the dilated ball $\tilde B= B(x, 3^{n+1}r)$ is contained in $V$.  We have already noted  in (\ref{covolume}) that $\cov(g_ku_{\f(x)}\Gamma) $ is the maximum of the norms of linear combinations of $1, f_1, \cdots, f_n$. Hence the first condition of Theorem \ref{Thm_sing} is satisfied. Thus we can conclude (\ref{e4}).
\subsubsection{\eqref{e4} $\implies$ \eqref{e11} in Theorems \ref{sing_main_1}} It is straightforward that \eqref{e4} $\implies$ \eqref{e11}.
\section{Parametrizing \texorpdfstring{$\cL$}{L} and its connection with singularity}
\label{parra}

One can define uniform exponent of a matrix $A$, following \cite{DFSU}.
\begin{definition}\label{omegahat}
Let $A$ be a $n\times m$ matrix. We define $\hat\omega(A)$ to be sup of all $\omega>0$ such that for all sufficiently large $Q>0$ there exist $\mathbf{0}\neq\bq\in \Z^m$, $\bp\in\Z^n$ satisfying the following system of inequalities, \begin{equation}\label{singAnomega}
 \begin{aligned}&\Vert A\bq +\bp\Vert<\frac{1}{Q^\omega},\\&\Vert \bq\Vert\leq Q.\end{aligned}
 \end{equation}
\end{definition}
We have the following lemma which connects $\hat\omega$ to $\omega$-singularity.
\begin{lemma}\label{omegahatsingular}
Let $A$ be a $k\times l$ matrix, and $n\in\N$. If $\hat\omega(A)=n+\delta$ for some $\delta>0$, then $A$ is $n$-singular. If $A$ is $n$-singular then $\hat\omega(A)\geq n.$ 
\end{lemma}
\begin{proof}
 Suppose $\hat\omega(A)=n+\delta$. Then for sufficiently large $Q>0$ there exist $\mathbf{0}\neq\bq\in \Z^l$, $\bp\in\Z^k$ satisfying the following system of inequalities, \begin{equation}\label{singAnomega}
 \begin{aligned}&\Vert A\bq +\bp\Vert<\frac{1}{Q^{n+\delta}},\\&\Vert \bq\Vert\leq Q.\end{aligned}
 \end{equation}
 For any $c>0$, and for sufficiently large $Q$, $\frac{1}{Q^\delta}\leq c$, hence by definition $A$ is $n$-singular.
 Now suppose $A$ is $n$-singular. Then by Definition \ref{omegahatsingular}, we know that for $c=1$ Equation \eqref{singAn} will have non zero integer solutions for all large $Q>0$. Hence $\hat\omega(A)\geq n$.
\end{proof}

\begin{lemma}\label{trivial}
If $A$ is a ${n\times 1} $ matrix such that $A$ is $n$-singular, then $A$ is in $\Q^n$.
\end{lemma}
\begin{proof}
 Let $A=\begin{bmatrix}
 \alpha_1\\
 \vdots\\
 \alpha_n
 \end{bmatrix}.$ If $A$ is $n$-singular then for every $c>0$ and all sufficiently large $Q>0$ there exist nonzero $q\in\Z$ and $\bp=(p_1,\cdots,p_n)\in\Z^n$ such that 
 $$ \max_{i=1}^n\vert q\alpha_i+p_i\vert \leq \frac{c}{Q^n}.$$
 The above implies each $\alpha_i$ is singular in $\R$. It is well known that singular numbers in $\R$ are just rationals. Hence each $\alpha_i\in\Q$ and we conclude this lemma. 
\end{proof}
\subsection{Parametrization}\label{para} 
If $\cL$ is an $s$-dimensional affine subspace of $\R^n$, one can choose a parametrization $\bx\to (\bx,\bx A_0+a_0)$,
where $A_0$ is a $s\times (n-s)$ matrix and $a_0\in \R^{n-s}$. Here both $\bx$ and $a_0$ are row vectors. We rewrite the parametrization above as $\bx\to (\bx,\tilde\bx A)$, where
$A=\begin{bmatrix}
a_0\\
A_0
\end{bmatrix}$ and $\tilde{\bx}=(1,\bx)$. We refer $\cL_{A}$ to be the affine subspace associated with the $s+1\times n-s$ matrix $A$.

\begin{lemma}\label{everything}
If $A$ is $n$-singular, then all $\by\in\cL$ are singular.
\end{lemma}
\begin{proof}
Since $A$ is $n$-singular, for all sufficiently large $Q>0$ there exist $\mathbf{0}\neq\bq\in \Z^{s+1}$, $\bp\in\Z^{n-s}$ satifying \eqref{singAn} for $\omega=n$.  Now for any $\bx\in\R^s$, one can write $$
\vert p_0+(\bx,\tilde{\bx}A)
\begin{bmatrix}
\bp'\\
\bq\\
\end{bmatrix}\vert \leq \Vert \tilde\bx\Vert \Vert A\bq+\bp\Vert,$$ where $\bp=(p_0,\cdots,p_s)\in\Z^{s+1}$ and $\bp'=(p_1,\cdots,p_s)\in\Z^s$. Since $\Vert \bp\Vert \leq C\Vert \bq\Vert$, where $C$ only depends on $A$, we conclude that any $(\bx,\tilde{\bx}A)\in\cL$ is singular.
\end{proof}


Let us denote $R_A:=\begin{bmatrix}
I_{s+1} & | & A
\end{bmatrix}$ where $A$ is a $s+1\times n-s$ matrix, and recall the definition of $\bc(\bw)$ in \eqref{cw}.
Let us recall Lemma 5.1 of \cite{Kleinbock-exponent}.
\begin{lemma}
Suppose that $\Vert R_A\bc(\bw)\Vert$ is less than $1$ for some $\bw\in \bigwedge^j\Z^{n+1}$. Then 
$$\Vert 
\bw\Vert\ll 1 + \Vert\pi(\bw)\Vert.$$
\end{lemma}

By the above lemma, Condition $(2)'$ \eqref{2prime} reduces to the following condition:
There exists $c>0$ ~
and  $~k_i\to\infty$  such that  $\forall~j=1,\cdots, n-s$ and  $\forall~\bw\in \bigwedge^{j}\Z^{n+1}$, one has  \begin{equation}\label{c2222}
\max\bigg( e^{(n-j+1)k_i}\Vert R_A\bc(\bw) \Vert, e^{-jk_i}\Vert \pi(\bw)\Vert\bigg)\geq c^{j},
\end{equation}
This is same as saying that:  

\textbf{Condition $(2)^\star$}\label{ccc}: There exists $c>0$ 
and  $Q_i\to\infty$  such that  $\forall~j=1,\cdots, n-s$ such that there is no nonzero $\bw\in \bigwedge^{j}\Z^{n+1}$, satisfying the following system:

\begin{equation}
\label{condition_lastest1}
\left\{\begin{aligned}
  &\Vert R_A\bc(\bw)\Vert<\frac{c^j}{Q^{(n-j+1)}_i}\\
  & \Vert \pi_{\bullet}(\bw)\Vert<c^j Q_i^j
\end{aligned}\right.
\end{equation}

Note that the above condition is a property of a subspace. The following theorem follows combining Theorem \ref{Thm_sing}, Propositions \ref{p2}, Proposition \ref{p3} and the above discussion.
\begin{theorem}\label{subspaceA}
Let $A$ and $\cL_{A}$ be as in this section. $\lambda_{\cL_{A}}$ almost every $\by\in\cL_{A}$ is not singular iff \textbf{Condition $(2)^\star$} (\ref{condition_lastest1}) is satisfied by $\cL_{A}$. 
\end{theorem}

\begin{proposition}\label{1isimportant}
 \textbf{Condition $(2)^\star$} (\ref{condition_lastest1}) for $j=1$ holds if and only if $A$ is not $n$-singular.
\end{proposition}
\begin{proof}
 Let $\bw=(\bq_0,\bq)\in \Z^{n+1}\setminus\{0\}$, $\bq_0\in \Z^{s+1}, \bq\in\Z^{n-s}$. Note that $\bc(\bw)=\bw$, $R_A\bc(\bw)=\bq_0+A\bq$ and $\pi_{\bullet}(\bw)=\bq$. Thus if $A$ is not $n$-singular then \textbf{Condition $(2)^\star$}  for $j=1$ is true.\end{proof}

\begin{lemma}\label{rowinterchange}
\textbf{Condition $(2)^\star$} (\ref{condition_lastest1})   does not change with row permutation of $A$.
\end{lemma}
\begin{proof}
Using Lemma 5.5 in \cite{Kleinbock-exponent}, we can conclude that the definition of $\Vert R_A\bc(\bw)\Vert$ does not change by permutting rows of $A$.
\end{proof}

\begin{lemma}
Let $A'=BA$, where $B\in \mathrm{GL}_{s+1}(\Q)$. Then $\cL_A$ satisfies \textbf{Condition $(2)^\star$} (\ref{condition_lastest1}) if and only if $\cL_{A'}$ satisfies \textbf{Condition $(2)^\star$} (\ref{condition_lastest1}). 
\end{lemma}
\begin{proof}
The proof follows from the observation in the proof of Lemma $5.6$ in \cite{Kleinbock-exponent}. It is observed that under the hypothesis of this lemma, for $\bw\in\bigwedge^j\Z^{n+1}$, there exists $\tilde\bw$ such that 
 $\Vert R_{A'}\bc(\tilde\bw)\Vert\leq C\Vert R_{A}\bc(\bw)\Vert$, $\Vert \bw\Vert=C_2\Vert\tilde{\bw}\Vert$, where $C_1, C_2$ depend on the matrices $A,B$ and $A'$. Using Lemma \ref{rowinterchange} and the previous observation we conclude this lemma.
\end{proof}
Let $V_0$ be the space spanned by $\be_1,\cdots,\be_n$ in $\R^{n+1}.$ We say that $A$ (or $\cL_{A}$) satisfies \textbf{Condition $(\omega,j)$} if there exists $c>0$ 
and  $Q_i\to\infty$ such that there is no nonzero $\bw\in \bigwedge^{j}\Z^{n+1}$, satisfying the following system:
\begin{equation}\label{condition_lastest}
\left\{\begin{aligned}
  &\Vert R_A\bc(\bw)\Vert<\frac{c^j}{Q^{\omega}_i}\\
  & \Vert \pi_{\bullet}(\bw)\Vert<c^j Q_i
\end{aligned}\right.
\end{equation}
\begin{proposition}\label{lastp}
Suppose that $A$ has more than one row, and let $A'$ be the matrix obtained from $A$ by removing one of its rows, and the removed row is a rational linear combination of remaining rows. $\cL_{A}$ satisfies \textbf{Condition $(\omega,j)$}  if and only if $\cL_{A'}$ satisfies \textbf{Condition $(\omega,j)$}. 
\end{proposition}
\begin{proof}
 Without loss of generality, we assume that the first row $\ba_0=\mathbf{0}$ is removed. For any $\bw\in \bigwedge^j\Z^{n+1}$, one observes that $\bc(\bw)_i=\bc(\pi(\bw))_i$ for $i>1$. Hence $$
R_A\bc(\bw)=\begin{bmatrix}
1 & \mathbf{0} & \ba_0\\
\mathbf{0}^T & I_{s} & A'\\
\end{bmatrix}\begin{bmatrix}
\bc(\bw)_0\\
\bc(\pi(\bw))_1\\
\vdots\\
\bc(\pi(\bw))_n\\
\end{bmatrix}=\begin{bmatrix}
\bc(\bw)_0\\
R_{A'}\bc(\pi(\bw))\\
\end{bmatrix}.$$ Therefore for any $\bw\in \bigwedge^j\Z^{n+1}$, 
$\Vert R_{A}\bc(\bw)\Vert=\max\{\Vert \bc(\bw)_0\Vert,\Vert R_{A'}\bc(\pi(\bw))\Vert\}$, $\pi_{\bullet}(\pi(\bw))=\pi_{\bullet}(\bw)$ and $\pi(\bw)\in \bigwedge^j\Z^n$. In particular, for $\bw\in \bigwedge^j V_0$, $\Vert  R_{A'}\bc(\pi(\bw))\Vert= \Vert R_{A}\bc(\bw)\Vert$. Hence if $\cL_{A}$ satisfies \textbf{Condition $(\omega,j)$}, then $\cL_{A'}$ satisfies the same.

Since $\Vert \bc(\bw)_0\Vert\geq 1 $ for any $\bw\in \bigwedge^j\Z^{n+1}$, unless it is $0$, we conclude  that if $\cL_{A'}$ satisfies \textbf{Condition $(\omega,j)$} (\ref{condition_lastest}), then $\cL_{A}$ satisfies the same.
\end{proof}

\subsection{Proof of Theorem \ref{main3}}
If $\cL=\{\bc\}$, an $0$ dimensional subspace Theorem \ref{main3} is true in a trivial manner. 

For affine hyperplanes $n-s=1$, \textbf{Condition $(2)^\star$} (\ref{condition_lastest1}) is only needed to be verified for $j=1$. By Proposition \ref{1isimportant} we conclude Theorem \ref{main3} for affine hyperplanes.
 
For subspaces considered in Theorem \ref{main3}, suppose the corresponding matrix $A$ is not $n$-singular. Then by Proposition \ref{1isimportant}, there exists $c>0$ and $Q_i\to\infty$ such that Equation \eqref{condition_lastest1} has no solution for $j=1$. That is \textbf{Condition $(n,1)$} is satisfied by $\cL_{A}$. Let us call the nontrivial row of $A$ to be $\ba\in \R^{n-s}$. By Propostion \ref{lastp}, \textbf{Condition $(n,1)$} is satisfied by $\cL_{\ba}=\{(0,\cdots,0,\ba)\}$. This means $(0,\cdots,0,\ba)\in\R^{n}$ is not $n$-singular, which implies $(0,\cdots,0,\ba)$ is not singular. Therefore, by Theorem \ref{subspaceA} we have that $\cL_{\ba}$ satisfies \textbf{Condition $(2)^\star$}(\ref{condition_lastest1}). Note that \textbf{Condition $(2)^\star$} is same as all of the  \textbf{Condition $(\frac{n-j+1}{j},j)$}, $j=1,\cdots,n$. Thus using Proposition \ref{lastp} we conclude that $\cL_{A}$ satisfies \textbf{Condition $(2)^\star$}. Hence we conclude by Theorem \ref{subspaceA}.\\

\subsection*{Acknowledgements} We thank Anish Ghosh and Ralf Spatzier for many helpful discussions and several helpful remarks which have improved the paper. We also thank Subhajit Jana for several helpful suggestions about the presentation of this paper. The second named author thanks University of Michigan for providing help as this project started as a Research Experience for Undergraduates project in summer 2021. 

\bibliographystyle{abbrv}
\bibliography{singular_real}

\end{document}